\documentclass[11pt,a4paper]{amsart}
\usepackage{hyperref}
\usepackage{graphicx} 
\usepackage{wasysym}
\usepackage[mathscr]{eucal}

\usepackage{amsmath}

\usepackage{amssymb}
\usepackage{amsthm}

\usepackage{amsfonts}
\usepackage{tikz}
\usepackage{enumerate}

\newtheorem{theorem}{Theorem}[section]
\newtheorem{lemma}[theorem]{Lemma}
\newtheorem{proposition}[theorem]{Proposition}
\newtheorem{corollary}[theorem]{Corollary}

\theoremstyle{definition}
\newtheorem{definition}[theorem]{Definition}
\newtheorem{example}[theorem]{Example}

\newtheorem{remark}[theorem]{Remark}

\begin{document}
	
	\title[Quasi-pseudometric modular spaces as $\mathscr{Q}$-categories]{Quasi-pseudometric modular spaces as $\mathscr{Q}$-categories}
	
	\author[C. L\'opez-Pastor]{C\'esar L\'opez-Pastor}
	\address[C. L\'opez-Pastor]{Departamento de Matemáticas, Universitat d'Alacant, Ctra. Sant Vicent del Raspeig, 03690 Alicante, Spain}
	\email{cesar.lopez@ua.es}
	
	\author[T. Pedraza]{Tatiana Pedraza}
	\address[T. Pedraza]{Instituto Universitario de Matem\'atica Pura y Aplicada, Universitat Polit\`ecnica de Val\`encia, Camino de Vera s/n, 46022 Valencia, Spain}
	\email{tapedraz@mat.upv.es}
	
	\author[J. Rodr\'{\i}guez-L\'opez]{Jes\'us Rodr\'{\i}guez-L\'opez}
	\address[J. Rodr\'{\i}guez-L\'opez]{Instituto Universitario de Matem\'atica Pura y Aplicada, Universitat Polit\`ecnica de Val\`encia, Camino de Vera s/n, 46022 Valencia, Spain}
	\thanks{The research of the two last authors is part of the project PID2022-139248NB-I00 funded by MICIU/AEI/10.13039/501100011033 and ERDF/EU }
	\email{jrlopez@mat.upv.es}

	
	
	\begin{abstract}
		We prove that the category of quasi-pseudometric modular spaces whose morphisms are the nonexpansive mappings is isomorphic to a quantale enriched category. To achieve this, we construct an appropriate quantale of isotone functions. We also show that, by means
		of this isomorphism, the topology associated with a quasi-pseudometric
		modular coincides with that generated by its corresponding quantale enriched
		category.
		
		Furthermore, we demonstrate that the class of quasi-pseudometrizable topological spaces coincides with the topological spaces whose topology is induced by a quasi-pseudometric modular. 
	\end{abstract}
	
	\date{This is the author's version of the article published in: Filomat 39, no. 19 (2025), 6693-6710. 
		The final published version is available via DOI at \href{https://doi.org/10.2298/FIL2519693L}{10.2298/FIL2519693L}.}

	\maketitle

	\section{Introduction}

	Nakano introduced the concept of modular \cite{Nakano50} to obtain a more detailed theory of Dedekind complete Riesz spaces and it was further extended to Riesz spaces and vector spaces.  A modular on a vector space is a nonnegative real-valued function, symmetric, convex, left-continuous, and non-identically null in each half-line. Its importance comes from the fact that you can construct a normed vector subspace from a modular, with the so-called Luxemburg norm \cite{KhamsiKoz15}. Moreover, modular spaces extend the Lebesgue, Riesz, and Orlicz spaces. 
	
	Recently, motivated by problems from multivalued analysis, Chistyakov \cite{Chis10a,Chis10b} introduced a general theory of modulars in arbitrary sets (removing the requirement of an algebraic structure in the underlying set) under the name of metric modular space. Roughly speaking, a metric modular space is a nonempty set endowed with a parameterized family $\{w_t\}_{t>0}$ of two-variable functions valued at $[0,+\infty]$ satisfying certain axioms that are consistent with the classical theory of modulars (see Definition \ref{def:qpmm}). The monograph \cite{Chis15} written by Chistyakov is a comprehensive study of the metric and topological properties of metric modular spaces. In particular, he introduced two different topologies in a metric modular space: the so-called metric topology and modular topology. The modular topology turns out to be the topologization  \cite[Theorem 4.3.5]{Chis15} of a non-topological convergence called \emph{modular convergence} \cite[Definition 4.2.1]{Chis15}, that extends the modular convergence defined by Musielak and Orlicz \cite{MusiOrlicz59}. 
	
	There also exists an asymmetric version of metric modular spaces, named quasi-pseudometric modular spaces, introduced by Sebogodi \cite{SebogodiPhD} in 2019, for which there is also a parallel theory to a certain extent.
	
	The purpose of this paper is to keep on exploring the theory of quasi-pseudometric modular spaces. Specifically, our objective is twofold. First, we aim to contribute to the basic theory of quasi-pseudometric modular spaces. This will be addressed in Section \ref{sec:qpmms}. After recalling the basic definitions, we introduce a quasi-uniformity (Proposition \ref{prop: W base uniformity}) on every quasi-pseudometric modular space, having as entourages the modular entourages considered by Chistyakov \cite[Section 4.1.2]{Chis15}. The topology generated by this quasi-uniformity is the quasi-pseudometric topology of the quasi-pseudometric modular. Moreover, we will show that the topology of a quasi-pseudometric space is also induced by a quasi-pseudometric modular (Theorem \ref{thm:qpm}). In addition, we analyze some concepts of functions between quasi-pseudometric modular spaces that can be considered as morphisms for the category of quasi-pseudometric modular spaces, which will be necessary for the second aim of the paper that we next discuss. 
	
	The study of the metrizability of a topological space has been one of the main research areas of general topology. Since not every topological space is metrizable, some authors have taken a different approach to this problem, searching for a more general concept of metric in such a way that every topology comes from a generalized metric. Quasi-pseudometrics (metrics that do not satisfy neither the symmetry axiom nor the non-degeneracy axiom) are probably the first generalized metrics but there are still topologies that are not quasi-pseudometrizable \cite{FL82}. 
	
	In 1978, Trillas and Alsina \cite{TrillasAlsina78} replaced the codomain of non-negative reals of a classic metric with an ordered algebraic structure. Kopperman tackled a similar approach \cite{Kopp88} in 1988, introducing the so-called continuity spaces by considering a value semigroup as the codomain of the metric. This afforded him to prove that every topological space is a continuity space.  Later on, Flagg \cite{Flagg97,FlaggKopp97} modified Kopperman's approach by evaluating a metric in a value quantale (see Definition \ref{def: value quantale} and sections \ref{sec:lq} and \ref{sec:Qc}) which provides important advantages with respect to the original continuity spaces (see \cite{CookWeiss21}). Furthermore, Flagg noticed that, in the same way that quasi-pseudometric spaces are enriched categories as first noticed by Lawvere \cite{Lawv73}, continuity spaces are just enriched categories over a value quantale. Roughly speaking, an enriched category is a generalization of the concept of a category where the set of morphisms are objects of a monoidal category. Therefore, the continuity spaces are $\mathscr{Q}$-categories \cite[Section III.1.3]{BookMonoidalTopology} where $\mathscr{Q}$ is a value quantale.
	
	The second goal of this paper is to demonstrate that the category of quasi-pseudometric modular spaces is isomorphic to a $\mathscr{Q}$-category for a concrete value quantale $\mathscr{Q}$. We will show this in Section \ref{sec:final} where we prove that the family $\nabla$ of all isotone functions between $(0,+\infty)$ and $[0,+\infty]^\text{\normalfont op}$ can be endowed with a specific order and operation that makes it a value quantale (Proposition \ref{prop:civalueq}). Then, we provide an isomorphism between the category of quasi-pseudometric modular spaces and the $\Delta$-category (Theorem \ref{thm:iso}). Furthermore, we show that this isomorphism also preserves the topologies of the objects (Theorem \ref{thm:topology}).
	These results establish the enriched category theory as a frame for studying quasi-pseudometric modular spaces that could allow for analyzing their relationship with other topological structures.

	\section{Quasi-pseudometric modular spaces}\label{sec:qpmms}
	
	We start by recalling the definition of a quasi-pseudometric modular \cite{SebogodiPhD}, the asymmetric version of the metric modular introduced by Chistyakov \cite{Chis10a,Chis15}.
	
	\begin{definition}[\cite{Chis15,SebogodiPhD}]\label{def:qpmm}
		Let $X$ be a nonempty set. A function $w:(0,+\infty)\times X\times X\to[0,+\infty]$ is  a \textbf{quasi-pseudometric modular} on $X$ if for every $x,y,z\in X$ and all $t,s>0$ it verifies:
		\begin{enumerate}[(M1)]
			\item\label{modular nondegenerate} $w(t,x,x)=0$ for all  $t>0;$
			\item\label{modular triangular} $w(t+s,x,y)\leq w(t,x,z)+w(s,z,y)$.
		\end{enumerate}
		
		\noindent If, in addition, $w$ satisfies
		\begin{enumerate}
			\item[(M3)] $w(t,x,y)=w(t,y,x)=0$ for all $t>0$ if and only if $x=y$
		\end{enumerate}
		then $w$ is called a \textbf{quasi-metric modular}.
		
		\noindent If a quasi-(pseudo)metric modular $w$ verifies
		\begin{enumerate}
			\item[(M4)] $w(t,x,y)=w(t,y,x)$ for all $x,y\in X$ and all $t>0$
		\end{enumerate}
		then $w$ is said to be a (pseudo)metric modular on $X.$

		The pair $(X,w)$ is known as a \textbf{(quasi)-(pseudo)metric modular space}.

		Moreover, a
		(quasi)-(pseudo)metric modular $w$ on $X$ is said to be \textbf{left-continuous} if $w(\cdot,x,y):(0,+\infty)\to[0,+\infty]$ is left-continuous for every $x,y\in X.$ In this case we say that $(X,w)$ is a \textbf{left-continuous (quasi)-(pseudo)metric modular space}.

	\end{definition}

	\begin{example}[\cite{Chis15}]\label{ex: canonical modular}
		Given a (quasi)-(pseudo)metric space $(X,d)$ and a nonincreasing function $g:(0,+\infty)\to [0,+\infty]$ non-identically zero,  then $w_g:(0,+\infty)\times X\times X\to[0,+\infty]$ defined as
		$$w_g(t,x,y)=g(t)\cdot d(x,y)$$
		for all $x,y\in X$ and all $t>0$, is a (quasi)-(pseudo)metric modular on $X.$
		
		If $g(t)=\tfrac{1}{t}$ for all $t>0$ then $w_g$ will be called the \textbf{standard (quasi)-(pseudo) metric modular induced by $d$} and will be denoted by $w_d$, that is,
		$$w_d(t,x,y)=\frac{d(x,y)}{t}$$
		for all $x,y\in X$ and all $t>.0$
		
	\end{example}
	
	One of the most important properties of a quasi-pseudometric modular, which can be deduced from (M\ref{modular triangular}), the triangular inequality, is the following:
	
	\begin{proposition}[\mbox{\cite[Lemma 3.1.1]{SebogodiPhD}}]\label{prop:isotone}
		Let $(X,w)$ be a quasi-pseudometric modular space. Then the function $w(\_,x,y):(0,+\infty)\to[0,+\infty]$ is non-increasing for all $x,y\in X$.
	\end{proposition}
	
	In \cite{Chis15}, Chistyakov considered two different topologies in a metric modular space that were later studied in the realm of quasi-pseudometric modular spaces in \cite{SebogodiPhD}: the metric topology and the modular topology. We provide here a new approach to the introduction of the metric topology by defining a quasi-uniformity from a quasi-pseudometric modular. 
	
	\begin{proposition}\label{prop: W base uniformity}
		Let $(X,w)$ be a quasi-pseudometric modular space. Given $t,\varepsilon>0$, define
		$$W^{w}_{t,\varepsilon}:=\{(x,y)\in X^2: w(t,x,y)<\varepsilon\}$$
		(we will omit the superscript $w$ if no confusion arises).
		\begin{enumerate}
			\item The family $\mathcal{B}=\{W^{w}_{t,\varepsilon}:t,\varepsilon>0\}$ is a base for a quasi-uniformity $\mathcal{W}_w$ on $X$. The elements $W^{w}_{t,\varepsilon}$ will be called \textbf{modular entourages}.
			\item $\left\{W^{w}_{\frac{1}{n},\frac{1}{n}}:n\in\mathbb{N}\right\}$ is a countable base for $\mathcal{W}_w.$
			\item If $w$ is a pseudometric modular, then $\mathcal{W}_w$ is a uniformity on $X.$
		\end{enumerate}
	\end{proposition}
	
	\begin{proof}
		We prove (1).
		
		By (M\ref{modular nondegenerate}), it is obvious that $\{(x,x):x\in X\}\subseteq W_{t,\varepsilon}$ for all $t,\varepsilon>0.$
		
		Let us see that $\mathcal{B}$ is a filter base.
		Given $t_1,t_2,\varepsilon_1,\varepsilon_2>0$, we claim that $W_{t_1\wedge t_2,\varepsilon_1\wedge \varepsilon_2}\subseteq W_{t_1,\varepsilon_1}\cap W_{t_2,\varepsilon_2}$. In fact if $(x,y)\in W_{t_1\wedge t_2,\varepsilon_1\wedge \varepsilon_2}$ then $w(t_1\wedge t_2,x,y)<\varepsilon_1\wedge \varepsilon_2.$
		Since $w(\_,x,y)$ is non-increasing then $\max\{w(t_1,x,y),w(t_2,x,y)\}\leq w(t_1\wedge t_2,x,y)<\varepsilon_1\wedge \varepsilon_2,$
		that is, $(x,y)\in W_{t_1,\varepsilon_1}\cap W_{t_2,\varepsilon_2}$.
		
		Now, by (M1), it is obvious that $\{(x,x):x\in X\}\subseteq W_{t,\varepsilon}$ for all $t,\varepsilon>0.$
		
		Let $t,\varepsilon>0$. Let us prove that $W_{\frac{t}{2},\frac{\varepsilon}{2}}\circ W_{\frac{t}{2},\frac{\varepsilon}{2}}\subseteq W_{t,\varepsilon}$. If $(x,y)\in W_{\frac{t}{2},\frac{\varepsilon}{2}}\circ W_{\frac{t}{2},\frac{\varepsilon}{2}}$, then there exists some $z\in X$ such that $(x,z),(z,y)\in W_{\frac{t}{2},\frac{\varepsilon}{2}}$, that is,
		\[\max\left\{w\left(\frac{t}{2},x,z\right),w\left(\frac{t}{2},z,x\right)\right\}<\frac{\varepsilon}{2}.\]
		By (M\ref{modular triangular}) we have that
		\[w(t,x,y)=w\left(\frac{t}{2}+\frac{t}{2},x,y\right)\leq w\left(\frac{t}{2},x,z\right)+w\left(\frac{t}{2},z,x\right)<\frac{\varepsilon}{2}+\frac{\varepsilon}{2}=\varepsilon.\]
		Hence, $(x,y)\in W_{t,\varepsilon}$ and $\mathcal{B}$ is a base of a quasi-uniformity on $X.$
		
		We next prove (2). Given some arbitrary $t,\varepsilon>0$ there exists $n_0\in\mathbb{N}$ such that $\frac{1}{n_0}<\min\{\varepsilon, t\}$. We claim that $W_{\frac{1}{n_0},\frac{1}{n_0}}\subseteq W_{t,\varepsilon}$.
		
		Take some $(x,y)\in W_{\frac{1}{n_0},\frac{1}{n_0}}$. Then $w\left(\frac{1}{n_0},x,y\right)<\frac{1}{n_0}$. Hence, since $w(\_,x,y)$ is non-increasing
		\[w(t,x,y)\leq w\left(\frac{1}{n_0},x,y\right)<\frac{1}{n_0}< \varepsilon,\]
		so $(x,y)\in W_{t,\varepsilon}$.
		
		Finally, to see (3), it is obvious that the modular entourages are symmetric in case that $w$ is a pseudometric modular. Thus, they form a base for a uniformity on $X$.
	\end{proof}
	
	\begin{definition}
		Let $(X,w)$ be a quasi-pseudometric modular space. The topology $\mathcal{T}(\mathcal{W}_w)$ generated by the quasi-uniformity $\mathcal{W}_w$ on $X$  will be called the \textbf{topology associated to the quasi-pseudometric modular} $w$. For simplicity, it will be also denoted by $\mathcal{T}(w)$.
		
		Then $\mathcal{T}(w)$ has as neighborhood base at $x\in X$ the family $\{W_{t,\varepsilon}(x):t,\varepsilon>0\}$ where
		$$W_{t,\varepsilon}(x)=\{y\in X:w(t,x,y)<\varepsilon\}.$$
		
	\end{definition}

	\begin{example}\label{ex:topology_sqpmm}
		Let $(X,d)$ be a quasi-pseudometric space. Consider the standard quasi-pseudometric modular $w_d$ on $X$ induced by $d$ (see Example \ref{ex: canonical modular}) given by  $$w(t,x,y)=\frac{d(x,y)}{t}$$ for all $x,y\in X$ and all $t>0.$ Then $\mathcal{T}(w)=\mathcal{T}(d)$. Let us check this.
		
		Notice first that for all $t,\varepsilon>0$ and $x\in X$,
		\[W_{t,\varepsilon}(x)=\left\{y\in X: w(t,x,y)=\frac{d(x,y)}{t}<\varepsilon\right\}=\{y\in X: d(x,y)<t\varepsilon\}=B(x,t\varepsilon).\]
		Hence, the neighborhood base at any $x\in X$ in $\mathcal{T}(w)$ coincides with all the open balls in $\mathcal{T}(d)$, so they generate the same topology.
	\end{example}

	\begin{remark}
		Given a pseudometric modular space $(X,w)$, Chistyakov \cite[Theorem 2.2.1]{Chis15} (see also \cite[Theorem 2.6]{Chis10a}) proved that the function $d_w(x,y):X\times X\to [0,+\infty)$ given by
		$$d_w(x,y)=\inf\{t>0:w(t,x,y)\leq t\}$$
		for all $x,y\in X,$ is an extended pseudometric on $X$ (i.e., a pseudometric that it is allowed to take the value $+\infty$). In case that $(X,w)$ is a quasi-pseudometric modular space, then $d_w$ is an extended quasi-pseudometric on $X$ \cite[Theorem 3.1.2]{SebogodiPhD}.
		
		Hence we can consider the open ball topology $\mathcal{T}(d_w)$ generated by the extended quasi-pseudometric $d_w$ on $X.$ For pseudometric modulars, Chistyakov \cite[Section 4.1.2]{Chis15} studied this topology, that he called \emph{metric topology}.
		The corresponding study in which $w$ is a quasi-pseudometric modular was performed in \cite{SebogodiPhD}.
		
		Moreover, in a quasi-pseudometric modular space $(X,w),$ we have that $\mathcal{T}(d_w)=\mathcal{T}(w)$ on $X,$ that is, the topology generated by the quasi-uniformity $\mathcal{W}_w$ is equal to the topology associated with the extended quasi-pseudometric $d_w.$ To see this, it suffices to observe that the quasi-uniformity $\mathcal{U}_{d_w}$ is in fact $\mathcal{W}_w$ since
		$$\{(x,y)\in X\times X:d_w(x,y)<\min\{t,\varepsilon\}\}\subseteq W_{t,\varepsilon},$$
		$$W_{\varepsilon,\varepsilon}\subseteq \{(x,y)\in X\times X:d_w(x,y)\leq \varepsilon\}.$$
	\end{remark}

	Therefore, the topology associated with a quasi-pseudometric modular is generated by a quasi-pseudometric, that is, it is quasi-pseudometrizable. But the converse is also true as we next show.

	\begin{definition}
		A topological space $(X,\mathcal{T})$ is said to be \textbf{quasi-pseudomodulable} if $\mathcal{T}=\mathcal{T}(w)$ for some quasi-pseudometric modular $w$ on $X$.
	\end{definition}

	\begin{theorem}\label{thm:qpm}
		A topological space is quasi-pseudomodulable if and only if is quasi-pseudo\-metrizable.
	\end{theorem}
	
	\begin{proof}
		Let $(X,\mathcal{T})$ be a quasi-pseudometrizable topological space. Then there exists some quasi-pseudometric $d$ on $X$ such that $\mathcal{T}=\mathcal{T}(d)$. By Example \ref{ex:topology_sqpmm}, $\mathcal{T}(d)=\mathcal{T}(w_d),$ where $w_d$ is the standard quasi-pseudometric modular associated with $d.$ Hence $\mathcal{T}$ is quasi-pseudomodulable.

		Conversely, suppose that there exists a quasi-pseudometric modular $w$ on $X$ such that $\mathcal{T}=\mathcal{T}(w)$. Since $\mathcal{T}(w)$ is induced by a quasi-uniformity $\mathcal{W}_w$ with a countable base, then it is quasi-pseudometrizable \cite{FL82}.
	\end{proof}
	
	We observe that, in general, the set $W_{t,\varepsilon}(x)$ is not open in $\mathcal{T}(w)$ even for metric modulars, as the next example shows.

	\begin{example}
		Let $X=\{x,y\}\cup\{z_n\}_{n\in\mathbb{N}}$ and define $w:(0,+\infty)\times X\times X\to[0,+\infty]$ as
		\begin{align*}
			w(t,a,a)&=0,\ \forall\; a\in X,\ \forall\; t>0.\\
			w(t,x,y)&=\begin{cases}
				1 &\text{ if } 0\leq t<1,\\
				0 &\text{ if } t\geq1,
			\end{cases}\\
			w(t,x,z_n)&=\begin{cases}
				1 &\text{ if } 0\leq t\leq1,\\
				0 &\text{ if } t>1,
			\end{cases}\\
			w(t,y,z_n)&=\begin{cases}
				\frac{1}{n} &\text{ if } 0\leq t<1,\\
				0 &\text{ if } t\geq1,
			\end{cases}\\
			w(t,z_n,z_m)&=\begin{cases}
				\frac{1}{\min\{n, m\}} &\text{ if } 0\leq t<1,\\
				0 &\text{ if } t\geq 1.
			\end{cases}\\
		\end{align*}
		It is straightforward to check that $(X,w)$ is a metric modular space.
		
		Let us see that $W_{1,\frac{1}{2}}(x)$ is not open in $\mathcal{T}(w).$ Since $w(1,x,y)=0$ then  $y\in W_{1,\frac{1}{2}}(x).$ We show that for all $t,\varepsilon>0$, $W_{t,\varepsilon}(y)\nsubseteq W_{1,\frac{1}{2}}(x)$ which shows that $W_{1,\frac{1}{2}}(x)$ is not open.

		Given any $t,\varepsilon>0$, there exists some $n_0\in\mathbb{N}$ such that $\frac{1}{n_0}<\varepsilon$. Thus,
		\[w(t,y,z_{n_0})=\left\{\begin{array}{ll}
			\frac{1}{n_0} &\text{ if }  0\leq t<1\\
			0 &\text{ if }  t\geq1
		\end{array}\right\}<\varepsilon.\]
		Hence, $z_{n_0}\in W_{t,\varepsilon}(y).$ Nevertheless,
		\[w(1,x,z_{n_0})=1>\frac{1}{2},\]
		which implies that $z_{n_0}\notin W_{1,\frac{1}{2}}(x)$. In conclusion, $W_{t,\varepsilon}(y)\nsubseteq W_{1,\frac{1}{2}}(x)$ for all $t,\varepsilon>0$.
	\end{example}
	
	Observe that in the previous example, $w(\_,x,y)$ is not left-continuous. This fact is not casual as it is inferred from the next result.
	
	\begin{proposition}
		Let $(X,w)$ be a left-continuous quasi-pseudometric modular space. Then $W_{t,\varepsilon}(x)$ is open for all $t,\varepsilon>0$ and for all $x\in X.$
	\end{proposition}
	
	\begin{proof}
		Let $x\in X,$ $t,\varepsilon>0,$ and $y\in W_{t,\varepsilon}(x)$. Define $\eta:=\varepsilon-w(t,x,y)>0$ and $t_n:=t-\frac{t}{2n},$ so $(t_n)_{n\in\mathbb{N}}$ converges to $t.$ Since $w(\_,x,y)$ is left-continuous then $(w\left(t_n,x,y\right))_{n\in\mathbb{N}}$ converges to $w(t,x,y)$. Hence there exists $n_0\in\mathbb{N}$ such that $|w(t_n,x,y)-w(t,x,y)|<\eta$ for all $n\geq n_0$. In particular
		\[w(t_{n_0},x,y)<\eta+w(t,x,y)=\varepsilon.\]
		Let us define $\delta:=\varepsilon-w(t_{n_0},x,y)$. We claim that $W_{t-t_{n_0},\delta}(y)\subseteq W_{t,\varepsilon}(x)$. Take $z\in W_{t-t_{n_0},\delta}(y)$. Then $w(t-t_{n_0},y,z)<\delta$ so
		\[w(t,x,z)=w(t-t_{n_0}+t_{n_0},x,z)\leq w(t_{n_0},x,y)+w(t-t_{n_0},y,z)<w(t_{n_0},x,y)+\delta=\varepsilon.\]
		Hence, $W_{t-t_{n_0},\delta}(y)\subseteq W_{t,\varepsilon}(x)$.
	\end{proof}

	Next, we study which morphisms can be considered between quasi-pseudo\-metric modular spaces to obtain an appropriate category.
	
	We first recall the following concept introduced in \cite{MonSintuKumam11,OOSebo22} for pseudometric modular spaces (see also \cite{DehEshaEba12}).

	\begin{definition}[\cite{MonSintuKumam11,OOSebo22}]
		A function $f:(X,w_1)\to (Y,w_2)$ between two quasi-pseudometric modular spaces is said to be \textbf{Lipschitz} if there exists $k>0$ such that
		$$w_2(k\cdot t,f(x),f(y))\leq w_1(t,x,y)$$
		for every $x,y\in X$ and every $t>0.$
		
		If $k=1$ then $f$ is called \textbf{nonexpansive}.
	\end{definition}
	
	\begin{remark}
		If the above condition is only satisfied when the parameter $t$ belongs to an interval $(0,t_0]$, then $f$ is called \emph{modular Lipschitzian} \cite{Chis15b}.
	\end{remark}
	
	We next introduce a new notion.

	\begin{definition}
		A function $f:(X,w_1)\to (Y,w_2)$ between two quasi-pseudometric modular spaces is said to be \textbf{strongly uniformly continuous} if given $t>0$ there exists $s>0$ such that
		$$w_2(t,f(x),f(y))\leq w_1(s,x,y)$$
		for every $x,y\in X.$
	\end{definition}
	
	\begin{proposition}
		Let $(X,w_1), (Y,w_2)$ be two quasi-pseudometric modular spaces. Each statement implies its successor:
		\begin{enumerate}[(1)]
			\item $f:(X,w_1)\to (Y,w_2)$ is Lipschitz;
			\item $f:(X,w_1)\to (Y,w_2)$ is strongly uniformly continuous;
			\item $f:(X,\mathcal{W}_{w_1})\to (Y,\mathcal{W}_{w_2})$ is uniformly continuous.
		\end{enumerate}
	\end{proposition}
	
	\begin{proof}
		(1) $\Rightarrow$ (2) By assumption, there exists $k>0$ such that
		$$w_2\left(k\cdot \frac{t}{k},f(x),f(y)\right)=w_2(t,f(x),f(y))\leq w_1\left(\frac{t}{k},x,y\right)$$
		for every $x,y\in X$ and every $t>0.$ Hence, $f$ is strongly uniformly continuous.
		
		(2) $\Rightarrow$ (3) Suppose that $f:(X,w_1)\to (Y,w_2)$ is strongly uniformly continuous. Let $V\in \mathcal{W}_{w_2}$. Then we can find $t,\varepsilon>0$ such that $W^{w_2}_{t,\varepsilon}\subseteq V.$
		By assumption, there exists $s>0$ such that
		$$w_2(t,f(x),f(y))\leq w_1(s,x,y)$$
		for every $x,y\in X.$ Hence if $(x,y)\in W^{w_1}_{s,\varepsilon}$ then $(f(x),f(y))\in  W^{w_2}_{t,\varepsilon}$ so  $f:(X,\mathcal{W}_{w_1})\to (Y,\mathcal{W}_{w_2})$ is uniformly continuous.
	\end{proof}

	Notice that for standard quasi-pseudometric modulars, Lipschitz functions are equal to strongly uniformly continuous functions. 
	
	\begin{proposition}
		Let $(X,d), (Y,q)$ be two quasi-pseudometric spaces. The following statements are equivalent:
		\begin{enumerate}[(1)]
			\item $f:(X,d)\to (Y,q)$ is Lipschitz;
			\item $f:(X,w_d)\to (Y,w_q)$ is Lipschitz;
			\item $f:(X,w_d)\to (Y,w_q)$ is strongly uniformly continuous.
		\end{enumerate}
	\end{proposition}
	
	\begin{proof}
		(1) $\Rightarrow$ (2) Since $f$ is Lipschitz there exists $k>0$ such that $q(f(x),f(y))\leq k\cdot d(x,y)$ for all $x,y\in X.$ Hence, for any $t>0$
		$$w_q(k\cdot t,f(x),f(y))=\frac{q(f(x),f(y))}{k\cdot t}\leq \frac{d(x,y)}{t}=w_d(t,x,y)$$
		which proves the statement.
		
		(2) $\Rightarrow$ (3) This follows from the previous Proposition.
		
		(3) $\Rightarrow$ (1). For $t=1$ we can find $s>0$ such that
		$$w_q(1,f(x),f(y))=q(f(x),f(y))\leq w_d(s,x,y)=\frac{d(x,y)}{s}$$
		for every $x,y\in X.$ Hence $f:(X,d)\to (Y,q)$ is Lipschitz with constant $\tfrac{1}
		{s}.$
	\end{proof}
	
	We denote by $\mathsf{QPMod}$ the category whose objects are
	the quasi-pseudometric modular spaces and whose morphisms are the strongly uniformly continuous maps. When we consider the nonexpansive maps as morphisms, we denote this category by $\mathsf{QPMod}_n$. Then $\mathsf{QPMod}_n$ is a subcategory of $\mathsf{QPMod}$.
	
	Moreover, we denote by $\mathsf{LQPMod}$ (resp. $\mathsf{LQPMod}_n$) the full subcategory of $\mathsf{QPMod}$ (resp. $\mathsf{QPMod}_n$) whose objects are the left-continuous quasi-pseudometric modular spaces. It turns out that $\mathsf{LQPMod}$  is a reflective subcategory of $\mathsf{QPMod}$.
	
	\begin{proposition}
		$\mathsf{LQPMod}$ is a reflective full subcategory of $\mathsf{QPMod}$ whose reflector is the functor $\mathscr{L}:\mathsf{QPMod}\to\mathsf{LQPMod}$ given by $\mathscr{L}((X,w))=(X,\widetilde{w})$ and leaving morphisms unchanged, where $\widetilde{w}$ is the left regularization of $w$ defined as
		$$\widetilde{w}(t,x,y)=\bigwedge_{0<s<t} w(s,x,y).$$
		for every $x,y\in X$ and every $t>0$ {\rm(}see \cite[Definition 1.2.4]{Chis15}{\rm)}.
	\end{proposition}
	
	\begin{proof}
		
		Following \cite[Proposition 1.2.5]{Chis15} we have that $\widetilde{w}$ is a left-continuous quasi-pseudometric modular on $X.$
		
		Moreover, let $f:(X,w_1)\to (Y,w_2)$ be a strongly uniformly continuous mapping. Given $t>0$ and $0<r<t$ there exists $s>0$ such that
		$$w_2(r,f(x),f(y))\leq w_1(s,x,y)$$
		for all $x,y\in X.$ Therefore,
		\begin{align*}
			\widetilde{w_2}(t,f(x),f(y))&=\bigwedge_{0<t'<t} w_2(t',f(x),f(y))\leq w_2(r,f(x),f(y))\\&\leq w_1(s,x,y)\leq \widetilde{w_1}(s,x,y)=\bigwedge_{0<s'<s} w_1(s',x,y)
		\end{align*}
		so $f:(X,\widetilde{w_1})\to (Y,\widetilde{w_2})$ is strongly uniformly continuous. Hence $\mathscr{L}$ is a functor.
		
		We next check that $\mathscr{L}$ is the left adjoint of the inclusion functor $\mathscr{I}:\mathsf{LQPMod}\to\mathsf{QPMod}.$ Let $(X,\omega_1)\in\mathsf{QPMod} $ and $(Y,w_2)\in \mathsf{LQPMod}$. Suppose that $f:(X,w_1)\to (Y,w_2)$ is strongly uniformly continuous. Given $t>0$ there exists $s>0$ such that $w_2(t,f(x),f(y))\leq w_1(s,x,y)$ for all $x,y\in X.$ Since $w_1(s,x,y)\leq \widetilde{w_1}(s,x,y)$ then $f:(X,\widetilde{w_1})\to (Y,w_2)$ is also strongly uniformly continuous.
		
		Now, let $g:(X,\widetilde{w_1})\to (Y,w_2)$ be strongly uniformly continuous. Given $t>0$ we can find $s>0$ such that
		$$w_2(t,f(x),f(y))\leq \widetilde{w_1}(s,x,y)=\bigwedge_{0<r<s} w_1(r,x,y).$$
		Hence $g:(X,w_1)\to (Y,w_2)$ is strongly uniformly continuous.
	\end{proof}
	
	\begin{remark}
		We observe that given $x,y\in X$, then $\widetilde{w}(\_,x,y)$ is the upper semicontinuous regularization or upper envelope of $w(\_,x,y)$, since this function is non-increasing (see \cite[Chapter 1.3]{BeerBook}).
	\end{remark}
	
	\begin{remark}\label{rem:functor}
		Observe that the above proof does not work using the categories $\mathsf{LQPMod}_n$ and $\mathsf{QPMod}_n$, although the same mapping between these two categories is still a functor.  
		For example, let $X$ be a set with at least two different points and consider the modular metric $x$ on $X$ given by
		$$w(t,x,y)=\begin{cases}
			0&\text{ if }x= y, t>0,\\
			1&\text{ if }x\neq y,0<t<1,\\
			0&\text{ if }x\neq y,t\geq 1,
		\end{cases}$$
		for all $x,y\in X$, $t>0.$ It is clear that its left regularization is
		$$\widetilde{w}(t,x,y)=\begin{cases}
			0&\text{ if }x= y, t>0,\\
			1&\text{ if }x\neq y,0<t\leq 1,\\
			0&\text{ if }x\neq y,t>1,
		\end{cases}$$
		for all $x,y\in X$, $t>0.$
		
		The identity map $i:X\to X$ is nonexpansive when $X$ is endowed with the metric modular $w$. However $i:(X,\widetilde{w})\to (X,w)$ is not nonexpansive since
		$$\widetilde{w}(1,x,y)=1\not\leq w(1,x,y)=0$$
		where $x,y$ are two distinct points of $X.$

	\end{remark}
	
	\section{Lattices and quantales}\label{sec:lq}
	
	The second goal of this paper is to establish an equivalence between the category of quasi-pseudometric modular spaces and a category enriched over a quantale (see Section \ref{sec:final}). Thus we need some preliminary concepts about order theory that will be useful later. Our main references for this section are \cite{GierzHoffKeiLawMisScott,BookMonoidalTopology,BookQuantales}.
	
	Recall that a \emph{partial order} $\leq$ on a nonempty set $X$ is a reflexive, antisymmetric, and transitive relation on $X.$ In this case, the pair $(X,\leq)$ is a \emph{partially ordered set} (a \emph{poset} for short). In this case, the opposite relation $\leq^\text{\normalfont op}$ given by
	$$x\leq^\text{\normalfont op} y\text{ if and only if }y\leq x$$
	for all $x,y\in X$, is also a partial order on $X.$ If no confusion arises, we will write $X^\text{\normalfont op}$ as short for $(X,\leq^\text{\normalfont op}).$
	
	A function $f:(X,\leq_1)\to (Y,\leq_2)$ between partially ordered sets is called \emph{isotone} if
	$$x\leq_1 y\text{ implies }f(x)\leq_2 f(y)$$
	for all $x,y\in X.$ The category of partially ordered sets with isotone maps as morphisms will be denoted by $\mathsf{POSet}.$

	Furthermore, a poset $(L,\leq)$ where every finite subset has an infimum and supremum is a \emph{lattice}. If every subset has an infimum and supremum, then it is a \emph{complete lattice}. If $A\subseteq L$ then $\bigvee A$, $\bigwedge A$ will denote the supremum and the infimum of $A$ respectively. If we want to emphasize the partial order that is used to compute the supremum or the infimum, we will write $\overset{{\scriptscriptstyle \leq}}{\bigvee},\overset{{\scriptscriptstyle\leq}}{\bigwedge}.$

	\begin{definition}\label{def: well below}
		Let $(L,\leq)$ be a complete lattice. Given $a,b\in L$, then $a$ is \textbf{well-below} $b$ ($a\lhd b$) if
		$$\text{for all } S\subseteq L\text{ such that }b\leq\bigvee S,\ \text{ there exists } s_0\in S\text{ such that } a\leq s_0.$$
	\end{definition}

	\begin{proposition}[Properties of the well-below order]\label{lema: well below}
		Let $(L,\leq)$ be a complete lattice.
		\begin{enumerate}
			\item $x\lhd y\Rightarrow x\leq y$.
			\item $x\lhd y\leq z$ or $x\leq y\lhd z$ implies $x\lhd z$.
			\item $\perp\lhd\, x$ if and only if $x\neq\perp.$
		\end{enumerate}
	\end{proposition}


	\begin{example}\label{ex: 01}
		In the complete lattice $([0,1],\leq)$, we have that $x\lhd y$ if and only if $x<y$. Let us check this. By Proposition \ref{lema: well below}, $x\lhd y$ implies $x\leq y$. Moreover $x\neq y$. Otherwise taking $S=\{s\in[0,1]: s<y\}$, we have that $y=\bigvee S$, although $s<x=y$ for all $s\in S$ which contradicts $x\lhd y$.
		
		\noindent Suppose now that $x<y$. Let $S\subseteq [0,1]$ such that $y\leq\bigvee S$. Then for all $\varepsilon>0$, there exists some $s_0\in S$ such that $y-\varepsilon\leq s_0$. Taking $\varepsilon=y-x>0$ we are done.
	\end{example}
	
	\begin{example}\label{ex: P(X)}
		Let us see that in $(\mathcal{P}(X),\subseteq)$, if $A,B\neq\varnothing$, then $A\lhd B$ if and only if  $A=\{b\}$ for some $b\in B$. Suppose that $A\lhd B$, then taking $\mathcal{S}=\{\{b\}\}_{b\in B}$ we have that $B=\bigvee \mathcal{S}=\bigcup \mathcal{S}$ and thus, there exists some $b_0\in B$ such that $A\subseteq \{b_0\}$, so $A=\{b_0\}$. 
		
		Conversely, suppose that $A=\{b\}$ for some $b\in B$. Let $\mathcal{S}\subseteq \mathcal{P}(X)$ such that $B\subseteq \bigvee\mathcal S$. Since $b\in B\subseteq \bigvee \mathcal S$, then there is some $S_0\in \mathcal S$ such that $b\in S_0$, which means that $A=\{b\}\subseteq S_0$. Hence $A\lhd B.$
		
		Observe that $\varnothing\lhd A$ for all $A\in\mathcal{P}(X).$
	\end{example}

	
		

	The proof of the following result is trivial so it is omitted.
	
	\begin{lemma}\label{lem: not trivial}
		Suppose that $(L,\leq)$ is a complete lattice. Then $\perp\lhd\top$ if and only if $L$ is not trivial.
	\end{lemma}
	
		
		
	
	
	\begin{definition}[\mbox{\cite[Definition I-2.8.]{GierzHoffKeiLawMisScott}}]
		A complete lattice $(L,\leq)$ is said to be \textbf{completely distributive} if given $\{a_{ij}: i\in I,\ j\in K(i)\}\subseteq L$ then
		\[\bigwedge_{i\in I}\bigvee_{j\in K(i)}a_{ij}=\bigvee_{f\in M}\bigwedge_{i\in I}a_{i,f(i)},\]
		where $M=\prod_{i\in I}K(i)$.
	\end{definition}
	
	\begin{theorem}[\cite{Raney53}] 
		A complete lattice $(L,\leq)$ is completely distributive if and only if $\forall b\in L$,
		\[b=\bigvee\{a\in L: a\lhd b\}.\]
	\end{theorem}
	
	\begin{example}\label{ex: 01 distributive}
		The complete lattice $([0,1],\leq)$ is completely distributive. By Example \ref{ex: 01}, $b=\bigvee\{a\in[0,1]: a<b\}=\bigvee\{a\in[0,1]: a\lhd b\}$.
	\end{example}
	
	\begin{example}\label{ex: P(X) distributive}
		The complete lattice $(\mathcal{P}(X),\subseteq)$ is completely distributive. By Example \ref{ex: P(X)}, $B=\bigcup_{b\in B}\{b\}=\bigcup\{A\subseteq X: A\lhd B\}$ for every nonempty set $B.$ On the other hand, by Proposition \ref{lema: well below} (3), $\{A\subseteq X: A\lhd\varnothing\}=\varnothing$ and this concludes our claim.
	\end{example}


	\begin{remark}\label{rem:product_completelydistributive}
		Notice that if $\{(L_\lambda,\leq_\lambda):\lambda\in \Lambda\}$ is an arbitrary family of completely distributive lattices then its Cartesian product $(\prod_{\lambda\in \Lambda} L_\lambda,\preceq)$ endowed with the componentwise partial order $\preceq$ is also completely distributive. This is clear since given $\{a_{ij}: i\in I,\ j\in K(i)\}\subseteq \prod_{\lambda\in \Lambda} L_\lambda$, then for every $\lambda\in \Lambda$
		\[\bigwedge_{i\in I}\bigvee_{j\in K(i)}a_{ij}(\lambda)=\bigvee_{f\in M}\bigwedge_{i\in I}a_{i,f(i)}(\lambda),\]
		where $M=\prod_{i\in I}K(i)$,
		since $(L_\lambda,\leq_\lambda)$ is completely distributive. As the supremum and infimum on $\prod_{\lambda\in\Lambda} L_\lambda$ is computed componentwisely then
		
		\[\bigwedge_{i\in I}\bigvee_{k\in K(i)}a_{ij}=\bigvee_{f\in M}\bigwedge_{i\in I}a_{i,f(i)},\]
		so $(\prod_{\lambda\in \Lambda} L_\lambda,\preceq)$ is completely distributive.
	\end{remark}

	\begin{definition}[\cite{Flagg97,FlaggKopp97}]\label{def:valuedl}
		A \textbf{value distributive lattice} is a completely distributive lattice  $(L,\leq)$ such that
		\begin{enumerate}[(VDL1)]
			\item\label{VL: bottom top} $\perp\lhd \top$
			\item\label{VL: supremum} $a,b\lhd\top\Rightarrow a\vee b\lhd\top$
		\end{enumerate}
	\end{definition}
	
	\begin{remark}
		Notice that: 
		\begin{itemize}
			\item  (VDL\ref{VL: bottom top}) just means that $L$ is not trivial by Lemma \ref{lem: not trivial}.
			\item (VDL\ref{VL: supremum}) just means that $\{a:a\lhd \top\}$ is directed.
		\end{itemize}
	\end{remark}

	\begin{example}
		Let $\mathbf{2}$ be the two element set $\{0,1\}$ endowed with the usual order $\leq.$ Then $(\mathbf{2},\leq)$ is a value distributive lattice.
	\end{example}
	
	\begin{example}\label{ex:0infinito}
		$([0,+\infty],\leq)$ (where $+$ represents the usual sum on the real numbers extended to $+\infty$ as usual) is a value distributive lattice.
	\end{example}
	
	We next introduce a crucial notion in our work: a quantale. This structure is a combination of order and a binary operation with some compatibility between them.
	
	\begin{definition}[\cite{BookQuantales}]
		A \textbf{quantale} is a triple $(\mathscr{Q},\preceq,\ast)$ where $(\mathscr{Q},\preceq)$ is a complete lattice and $\ast$ is a binary operation on $\mathscr{Q}$ such that
		\begin{enumerate}[(q1)]
			\item\label{quantal: semigroup} $(\mathscr{Q},\ast)$ is a semigroup.
			\item\label{quantal: right distributive} $a\ast(\bigvee_{i\in I}b_i)=\bigvee_{i\in I}(a\ast b_i)$.
			\item\label{quantal: left distributive} $(\bigvee_{i\in I}b_i)\ast a=\bigvee_{i\in I}(b_i\ast a)$.
		\end{enumerate}
		
		\noindent where $\{b_i:i\in I\}\subseteq\mathscr{Q}$  and $a\in\mathscr{Q}.$ \medskip
		
		\noindent A quantale $(\mathscr{Q},\preceq,\ast)$ is said to be:
		\begin{itemize}
			\item \textbf{commutative} if $\ast$ is commutative;
			\item \textbf{unital} if  $(\mathscr{Q},\ast)$ is a monoid with unit $1_\mathscr{Q};$
			\item  \textbf{integral} if it is unital and the unit is the top element of $(\mathscr{Q},\preceq)$, that is, $1_\mathscr{Q}=\top_{\mathscr{Q}}.$ If no confusion arises we will simply write $\top$ instead of $\top_{\mathscr{Q}}.$
		\end{itemize}	
	\end{definition}
	
	In the remainder of the paper, we will refer to commutative integral quantales as CI-quantales. Moreover, if no confusion arises, we will denote a quantale $(\mathscr{Q},\preceq,\ast)$ only by its underlying set $\mathscr{Q}.$

	Notice that in an integral quantale $(\mathscr{Q},\preceq,\ast)$ we have that $u\ast v\preceq u\wedge v$ for all $u,v\in \mathscr{Q}.$ In fact, $u=u\ast (v\vee \top)=(u\ast \top)\vee (u\ast v)=u\vee (u\ast v)$ so $u\ast v\preceq u.$ In a similar way, $u\ast v\preceq v.$
	
	\begin{example}
		$(\mathbf{2},\leq,\wedge)$ is a CI-quantale.
	\end{example}
	
	\begin{example}\label{ex: suma reales}
		$([0,+\infty],\leq,+)$ is a commutative unital quantale, but it is not integral since its unit is $0\neq\top=+\infty.$
		
		On the other hand, $\mathsf{P}_+=([0,+\infty],\leq^\text{\normalfont op},+)$ is a CI-quantale. This quantale is sometimes called the Lawvere quantale \cite{CookWeiss22} (see also \cite[Example II.1.10.1.(3)]{BookMonoidalTopology}).
	\end{example}

	\begin{example}
		Let $X$ be a nonempty set and $(\mathscr{Q},\preceq,\ast)$ be a quantale. Then we can endow the set $\mathscr{Q}^X$ of all maps $f:X\to\mathscr{Q}$ with the pointwise order that for simplicity we also denote by $\preceq.$  Then $(\mathscr{Q}^X,\preceq)$ is also a complete lattice (see for example \cite[Example 2.1.9]{BookQuantales}). Notice that meet and joins in $\mathscr{Q}^X$ are computed pointwisely.
		
		Moreover, defining a binary operation on $\mathscr{Q}^X$ pointwisely by means of $\ast$, that we again denote by $\ast$, turns $(\mathscr{Q}^X,\preceq,\ast)$ into a quantale.
		
		Furthermore, if $X$ is not only a set but also a partially ordered set, then the family $\mathcal{I}(\mathscr{Q}^X)$  of all the isotone maps between $X$ and $\mathscr{Q}$ is a sublattice of $\mathscr{Q}^X$ which is also a quantale.
		
	\end{example}
	
	%
	%
	%
	%
			%

	\begin{example}
		A complete lattice $(X,\preceq)$ such that $(X,\preceq,\wedge)$ is a quantale, is called a \textbf{complete Heyting algebra} or a \textbf{frame} \cite{GierzHoffKeiLawMisScott}.
		
		In particular, a topology $\mathcal{T}$ on a nonempty set $X$ has a quantale structure $(\mathcal{T},\subseteq,\cap).$
	\end{example}
	
	The following concept was introduced in \cite{Flagg97} to obtain a generalization of the notion of  a metric, as it replicates the essential properties of $[0,+\infty],$ the codomain of an extended metric.
	
	\begin{definition}[\cite{Flagg97,FlaggKopp97,CookWeiss21}]\label{def: value quantale}
		A \textbf{value quantale} is a quantale $(\mathscr{Q},\preceq,\ast)$ such that $(\mathscr{Q},\preceq)$ is a value distributive lattice.
	\end{definition}

	\begin{example}
		$(\mathbf{2},\leq,\wedge)$ is a value quantale.
	\end{example}
	
	\begin{example}
		$([0,1],\leq,\cdot)$ is a value quantale. It is obvious that it is a quantale. Notice that $\lhd$ is precisely $<$, so it immediately follows that it is a value quantale.     
	\end{example}
	
	\begin{example}
		$(\mathcal{P}(X),\subseteq,\bigcap)$ is a quantale but not a value quantale if $|X|>1$. Let $x,y\in X$ be two different points. By Example \ref{ex: P(X)}, $\{x\}\lhd X$ and $\{y\}\lhd X$. Nevertheless, $\{x\}\vee \{y\}=\{x,y\}\ntriangleleft X$ again by Example \ref{ex: P(X)}, so  (VDL2) does not hold.
	\end{example}

	\section{$\mathscr{Q}$-categories}\label{sec:Qc}
	
	As we have mentioned in the introduction, one of our main goals is to present quasi-pseudometric modular spaces as a particular example of a $\mathscr{Q}$-category, that is, an enriched category over a commutative unital quantale (see \cite{BookEnrichedCategory,BookMonoidalTopology}). This will be developed in the next section but we first present a summary of the notions that will be needed.

	\begin{definition}[\mbox{\cite[Section III.1.3]{BookMonoidalTopology}, c.f. \cite[Definition 3.1]{FlaggKopp97}}]
		Let $(\mathscr{Q},\preceq,\ast)$ be a commutative unital quantale. A \emph{$\mathscr{Q}$-category} is a pair $(X,q)$ where $X$ is a nonempty set and $q:X\times X\to\mathscr{Q}$ is a map such that: 
		\begin{enumerate}[(QC1)]
			\item\label{QC: qxx} $\top\preceq q(x,x),$
			\item\label{QC: triangular} $q(x,z)*q(z,y)\preceq q(x,y),$
		\end{enumerate}
		for all $x,y,z\in X.$
		
		\noindent  A $\mathscr{Q}$-functor is a map $f:(X,a)\to (Y,b)$ between $\mathscr{Q}$-categories such that
		$$a(x,y)\preceq b(f(x),f(y))$$
		for every $x,y\in X.$
		
		\noindent  $\mathscr{Q}$-categories and $\mathscr{Q}$-functors form a category denoted by $\mathscr{Q}$-$\mathsf{Cat}.$
	\end{definition}

	\begin{definition}
		A $\mathscr{Q}$-category $(X,q)$ is said to be:
		\begin{itemize}
			\item \emph{separated} if given $x,y \in X$, whenever $\top\preceq q(x,y)$ and $\top\preceq q(y,x)$ then $x=y.$
			\item \emph{symmetric} if $q(x,y)=q(y,x)$ for all $x,y\in X.$
		\end{itemize}
	\end{definition}
	
	We next provide several well-known examples of $\mathscr{Q}$-categories \cite{Flagg97,BookMonoidalTopology}.
	
	\begin{example}[\mbox{\cite[Example III.1.3.1.(1)]{BookMonoidalTopology}}]
		$\mathbf{2}$-categories and preordered sets are equivalent concepts.
		If $(X,a)$ is a $\mathbf{2}$-category, the binary relation $\preceq_a$ on $X$ given by $$x\preceq_a y\Leftrightarrow a(x,y)=1$$
		is a preorder on $X.$ A similar argument allows to convert a preordered set $(X,\preceq)$ into a $\mathbf{2}$-category.
		
		Furthermore, a $\mathbf{2}$-functor between two $\mathbf{2}$-categories $(X,a)$ and $(Y,b)$ is an isotone function between the preordered sets $(X,\preceq_a)$ and $(Y,\preceq_b).$ So $\mathbf{2}$-Cat is isomorphic to the category of preordered sets and isotone maps.
		
	\end{example}
	
	\begin{example}[\mbox{\cite[Example III.1.3.1.(2)]{BookMonoidalTopology}}]\label{ex:msec}
		$\mathsf{P}_+$-categories (see Example \ref{ex: suma reales}) are extended quasi-pseudometric spaces.
		
		If $(X,a)$ is a $\mathsf{P}_+$-category then $a:X\times X\to [0,+\infty]$ is a map verifying
		$$0\geq a(x,x)\hspace*{1cm}\text{and}\hspace*{1cm}a(x,z)+a(z,y)\geq a(x,y)$$
		for all $x,y,z\in X$, so $(X,a)$ is an \textbf{extended quasi-pseudometric space} \cite{KunziCh01} or an hemi-metric space \cite{GL13}.

		Moreover, a $\mathsf{P}_+$-functor between two $\mathsf{P}_+$-categories $(X,a)$, $(Y,b)$ is map $f:(X,a)\to (Y,b)$ verifying
		$$a(x,y)\geq b(f(x),f(y))$$
		for all $x,y\in X,$ that is, a nonexpansive mapping between the extended quasi-pseudometric spaces $(X,a)$, $(Y,b)$.
		
		Thus, $\mathsf{P}_+$-$\mathsf{Cat}$ is exactly the category $\mathsf{EQPMet}$ of extended quasi-pseudometric spaces.
		
	\end{example}

	In \cite{Flagg97} (see also \cite{CookWeiss21}), Flagg introduced a topology in a continuity space, that is, a $\mathscr{Q}$-category where $\mathscr{Q}$ is a value quantale. This topology was inspired by the classic open ball topology of a metric space, and the topology of the original continuity spaces of Kopperman \cite{Kopp88}, where metrics are valued in what he called a \emph{value semigroup} (see \cite{CookWeiss21}). This topology is important since it allows to prove that every topology comes from a $\mathscr{Q}$-category for a certain value quantale $\mathscr{Q}$ (\cite{CookWeiss21,Flagg97}). We recall the definition of this topology.
	
	\begin{definition}[\cite{Flagg97,FlaggKopp97}]
		Let $(X,a)$ be a $\mathscr{Q}$-category, with $\mathscr{Q}$ being a value quantale. Given $x\in X$ and $r\in\mathscr{Q}$ such that $r\lhd\top$, the \textbf{open ball} centered in $x$ with radius $r$ is defined as
		\[B(x,r):=\{y\in X: r\lhd a(x,y)\}.\]
	\end{definition}
	
	\begin{proposition}[\cite{Flagg97,FlaggKopp97}]\label{prop: base open balls quantale}
		Let $(X,a)$ be a $\mathscr{Q}$-category, with $\mathscr{Q}$ being a value quantale. Then $\{B(x,r): r\lhd\top,\ x\in X\}$ is a base for a topology $\mathcal{T}(a)$ on $X.$
	\end{proposition}
	
	\color{black}
	
	\section{Quasi-pseudometric modular spaces as $\mathscr{Q}$-categories}\label{sec:final}
	
	In this section, we address the main goal of the paper: to establish an equi\-valence between quasi-pseudometric modular spaces and certain $\mathscr{Q}$-categories. To achieve this, we need to consider a particular quantale that we define using the next few results.
	
	\begin{lemma}\label{lem:deltacd}
		Let us consider the set
		$$\nabla:=\Big\{f:(0,+\infty)\to[0,+\infty]^\text{\normalfont op}\text{ such that } f\text{ is isotone}\Big\}$$ endowed with the pointwise order induced by the order $\leq^\text{\normalfont op}$ on the codomain $[0,+\infty]$, that is,
		\[f\leq^\text{\normalfont op}g\Leftrightarrow f(t)\leq^\text{\normalfont op}g(t),\ \text{ for all } t\in(0,\infty).\]
		Then $(\nabla,\leq^\text{\normalfont op})$ is a completely distributive lattice where the top and bottom elements are
		the constant 0 function denoted by $\mathbf{0}$, and the constant $\infty$ function denoted by $\boldsymbol{\infty}$, respectively.
	\end{lemma}
	
	\begin{proof}
		
		It is straightforward to verify that $(\nabla,\leq^\text{\normalfont op})$ is a partially ordered set. Moreover, it is easy to check that the supremum and infimum in $\nabla$ are computed pointwisely, that is, if $F\subseteq\nabla$ then
		\begin{align*}
			\left(\overset{\quad{\scriptscriptstyle\leq^{\text{\normalfont op}}}}{\bigvee}  F\right)(t)&=\overset{\quad{\scriptscriptstyle\leq^{\text{\normalfont op}}}}{\bigvee} \{f(t):f\in F\}\\
			\left(\overset{\quad{\scriptscriptstyle\leq^{\text{\normalfont op}}}}{\bigwedge} F\right)(t)&=\overset{\quad{\scriptscriptstyle\leq^{\text{\normalfont op}}}}{\bigwedge} \{f(t):f\in F\}
		\end{align*}
		for all $t>0$. Hence $(\nabla,\leq^\text{\normalfont op})$ is a complete lattice.
		
		Moreover,  $([0,+\infty]^{(0,+\infty)},\leq^\text{\normalfont op})$ is completely distributive since it is the Cartesian product of completely distributive lattices (see Remark \ref{rem:product_completelydistributive}.) Since $(\nabla,\leq^\text{\normalfont op})$ is a sublattice of $([0,+\infty]^{(0,+\infty)},\leq^\text{\normalfont op})$ then it is completely distributive.
	\end{proof}

	\begin{proposition}
		Let
		$$\nabla_{L}=\left\{f:(0,+\infty)\to[0,+\infty]^\text{\normalfont op}\text{ such that } f\text{ is isotone} \text{ and } f(t)=\bigvee^{\leq^\text{\normalfont op}}_{0<s<t} f(s)\right\}.$$
		Then $(\nabla_L,\leq^\text{\normalfont op})$ is a complete sublattice of $(\nabla,\leq^\text{\normalfont op}).$
	\end{proposition}

	\begin{proof}
		We only prove that the supremum of a family $\mathcal{F}\subseteq\nabla_L$ belongs to $\nabla_L$ and that this supremum is computed pointwisely. In this way, let us define $F:(0,+\infty)\to[0,+\infty]^\text{\normalfont op}$ as $F(t)=\bigvee^{{\leq^\text{\normalfont op}}}\{f(t): f\in\mathcal{F}\}$ for all $t\in (0,+\infty)$.
		
		It is obvious that $F$ is isotone. Moreover, let $t\in (0,+\infty).$
		Since $F$ is isotone then $\bigvee^{{\leq^\text{\normalfont op}}}_{0<s<t} F(s)\leq^\text{\normalfont op} F\left(t\right).$ On the other hand, for all $f\in\mathcal{F}$ we have that $f(t)=\bigvee^{{\leq^\text{\normalfont op}}}_{0<s<t} f(s)$ so
		\begin{align*}
			F(t)&=\bigvee^{{\leq^\text{\normalfont op}}}\{f(t): f\in\mathcal{F}\}=\bigvee^{{\leq^\text{\normalfont op}}}\Big\{\bigvee^{{\leq^\text{\normalfont op}}}_{0<s<t} f(s): f\in\mathcal{F}\Big\}\leq  \bigvee^{{\leq^\text{\normalfont op}}}_{0<s<t}\Big\{\bigvee^{{\leq^\text{\normalfont op}}} f(s): f\in\mathcal{F}\Big\}\\
			&=\bigvee^{{\leq^\text{\normalfont op}}}_{0<s<t} F(s).
		\end{align*}
		
		Therefore, $F\in\nabla_L.$
	\end{proof}

	To prove that $(\nabla,\leq^\text{\normalfont op})$ is a value distributive lattice, the following lemma will be useful.
	
	\begin{lemma}\label{lem:wbnabla}
		Let $\boldsymbol{\infty}\neq f\in \nabla$. Then $f\lhd^\text{\normalfont op} \mathbf{0}$ if and only if
		\begin{enumerate}[(1)]
			\item
			there exists $s\in(0,+\infty)$ such that $f(t)=\infty$ for every $t\in (0,s)$ and
			\item $\bigvee^{\leq^\text{\normalfont op}} \{f(t):t\in (0,+\infty)\}<^\text{\normalfont op}0.$
		\end{enumerate}
	\end{lemma}
	
	\begin{proof}
		Suppose that $f\lhd^\text{\normalfont op} \mathbf{0}.$
		
		We first prove (1). Suppose that $f(t)\neq\infty$ for every $t\in (0,+\infty).$ For each $n\in\mathbb{N}$, let $g_n:(0,+\infty)\to [0,+\infty]$ defined as
		$$g_n(t)=\begin{cases}
			f(t)+1&\text{ if }0<t<\frac{1}{n},\\
			0&\text{ if }\frac{1}{n}\leq t.
			
		\end{cases}$$
		It is obvious that $g_n\in\nabla$ for every $n\in\mathbb{N}$ and $\bigvee^{\leq^\text{\normalfont op}} g_n=\mathbf{0}.$ However, $f\not\leq^\text{\normalfont op} g_n$ for every $n\in\mathbb{N}$, which contradicts $f\lhd^\text{\normalfont op} \mathbf{0}.$
		
		Consequently, there exists $t_0\in (0,+\infty)$ such that $f(t_0)=+\infty.$ Since $f$ is isotone then $f(t)=+\infty$ for very $t\leq t_0.$ Define $s:=\bigvee\{t\in (0,+\infty):f(t)=\infty\}<\infty$ since $f\neq\boldsymbol{\infty}.$ Therefore $f(t)=\infty$ for every $t\in (0,s).$
		
		We next prove (2). Suppose that $\bigvee^{\leq^\text{\normalfont op}} \{f(t):t\in (0,+\infty)\}=0.$ For every $n\in\mathbb{N}$, let $h_n:(0,+\infty)\to [0,+\infty]$ defined as $$h_n(t)=\frac{1}{n}\text{ for all }t\in (0,+\infty).$$ Obviously, $h_n\in\nabla$ for every $n\in\mathbb{N}$ and $\bigvee^{\leq^\text{\normalfont op}}_{n\in\mathbb{N}} h_n=\mathbf{0}.$ However, given $n\in\mathbb{N}$ we can find $t_n\in (0,+\infty)$ such that $h_n(t_n)=\frac{1}{n}<^\text{\normalfont op}f(t_n)\leq^\text{\normalfont op} 0.$ Hence $f\not\leq^\text{\normalfont op} h_n$ for all $n\in\mathbb{N}$, which contradicts $f\lhd^\text{\normalfont op}\mathbf{0}.$
		
		Conversely, let $s\in(0,+\infty)$ such that $f(t)=\infty$ for very $t\in (0,s)$ and let $a=\bigvee^{\leq ^\text{\normalfont op}} \{f(t):t\in (0,+\infty)\}<^\text{\normalfont op} 0$. Let us prove that $f\lhd^\text{\normalfont op} \mathbf{0}$. Let $\mathcal{F}\subseteq \nabla$ such that $\bigvee^{\leq^\text{\normalfont op}} \mathcal{F}=\mathbf{0}.$ Then we can find $g\in \mathcal{F}$ such that $a\leq^\text{\normalfont op} g(s).$ Now, if $t<s$, then by hypothesis (1), $f(t)=\infty\leq^\text{\normalfont op}g(t)$. On the other hand, if $t\geq s$, then by hypothesis (2) and the fact that $g$ is isotone, $f(t)\leq^\text{\normalfont op}a\leq^\text{\normalfont op}g(s)\leq^\text{\normalfont op}g(t)$.
	\end{proof}

	\begin{corollary}
		$(\nabla,\leq^\text{\normalfont op})$ is a value distributive lattice.
	\end{corollary}
	
	\begin{proof}
		We already know by Lemma \ref{lem:deltacd} that $(\nabla,\leq^\text{\normalfont op})$ is completely distributive. 
		
		Moreover, it is obvious that $\boldsymbol{\infty}\lhd^\text{\normalfont op}\mathbf{0}.$
		
		Consider $f,g\in\nabla$ such that $f\lhd^\text{\normalfont op}\mathbf{0}$ and $g\lhd^\text{\normalfont op}\mathbf{0}.$ By the previous lemma, we can find $s_f,s_g\in (0,+\infty)$ such that $f(t)=\infty$ for every $t\in (0,s_f)$ and $g(t)=\infty$ for every $t\in (0,s_g)$. Hence $(f\vee^{\leq^\text{\normalfont op}} g)(t)=\infty$ for every $t\in (0,s_f\wedge s_g)$ so $f\vee^{\leq^\text{\normalfont op}} g$ verifies condition (1) of the above lemma.
		
		Furthermore, $\bigvee^{\leq^\text{\normalfont op}} \{f(t):t\in (0,+\infty)\}<^\text{\normalfont op}0$ and $\bigvee^{\leq^\text{\normalfont op}} \{g(t):t\in (0,+\infty)\}<^\text{\normalfont op}0$ which obviously implies $\bigvee^{\leq^\text{\normalfont op}} \{(f\vee^{\leq^\text{\normalfont op}} g)(t):t\in (0,+\infty)\}<^\text{\normalfont op}0$. By the preceding lemma, $f\vee^{\leq^\text{\normalfont op}} g\lhd^\text{\normalfont op}\mathbf{0}$. Consequently, $(\nabla,\leq^\text{\normalfont op})$ is a value distributive lattice.
	\end{proof}
	
	\begin{proposition}\label{prop:civalueq}
		Consider the binary operation $\oplus:\nabla\times\nabla\to\nabla$ given by
		\[(f\oplus g)(t):=\bigvee_{r+s\leq t}^{\leq^\text{\normalfont op}}(f(r)+g(s))=\bigvee_{r+s=t}^{\leq^\text{\normalfont op}}(f(r)+g(s))\]
		for all $t>0.$ 
		Then $(\nabla,\leq^\text{\normalfont op},\oplus)$ and $(\nabla_L,\leq^\text{\normalfont op},\oplus)$ are  CI-value quantales.
	\end{proposition}
	
	\begin{proof}
		
		It is straightforward to check that $f\oplus g\in\nabla$ for every $f,g\in\nabla.$

		Given $t>0,$ we prove that
		$$\bigvee_{r+s=t}^{\leq^\text{\normalfont op}}f(r)+g(s)=\bigvee_{r+s\leq t}^{\leq^\text{\normalfont op}}f(r)+g(s).$$

		Notice that $\{f(r)+ g(s): r+s=t\}\subseteq \{f(r)+g(s): r+s\leq t\}$, so
		\[\bigvee_{r+s=t}^{\leq^\text{\normalfont op}}(f(r)+g(s))\leq\bigvee_{r+s\leq t}^{\leq^\text{\normalfont op}}(f(r)+ g(s))=(f\oplus g)(t).\]
		On the other hand, if $r+s\leq t$, then $s\leq t-r$. By isotonicity of $g$, $g(s)\leq^\text{\normalfont op} g(t-r)$ so $$f(r)+g(s)\leq^\text{\normalfont op} f(r)+g(t-r)\leq^\text{\normalfont op}\bigvee_{r\leq t}^{\leq^\text{\normalfont op}}(f(r)+ g(t-r))=\bigvee_{r+s=t}^{\leq^\text{\normalfont op}}(f(r)+g(s)).$$ Hence,
		\[(f\oplus g)(t)=\bigvee_{r+s\leq t}^{\leq^\text{\normalfont op}}(f(r)+g(s))\leq^\text{\normalfont op} \bigvee_{r+s=t}^{\leq^\text{\normalfont op}}(f(r)+g(s)),\]
		proving the desired equality.
		
		We next check that $(\nabla,\oplus)$ is a commutative monoid.
		
		First, commutativity is clear from the commutativity of the sum. Furthermore, given $f\in\nabla$,
		\[(f\oplus\mathbf{0})(t)=\bigvee_{r+s=t}^{\leq^\text{\normalfont op}}f(r)+\mathbf{0}(s)=\bigvee_{r+s=t}^{\leq^\text{\normalfont op}}f(r)=f(t),\]
		where the last inequality holds since $f$ is isotone. So $\mathbf{0}$ is the neutral element for $\oplus.$

		Finally, to prove the associativity property, one has that
		\begin{align*}
			((f\oplus g)\oplus h)(t)&=\bigvee_{r+s=t}^{\leq^\text{\normalfont op}}(f\oplus g)(r)+h(s)=\bigvee_{r+s=t}^{\leq^\text{\normalfont op}}\left(\bigvee_{u+v=r}f(u)+g(v)\right)+h(s)=\\			&=\bigvee_{\substack{r+s=t \\ u+v=r}}^{\leq^\text{\normalfont op}}f(u)+g(v)+ h(s)=\bigvee_{u+v+s=t}^{\leq^\text{\normalfont op}}f(u)+ g(v)+ h(s).
		\end{align*}
		where the last equalities hold by the continuity of the sum. Since the final expression does not depend on the order of the elements, we can assure that $\oplus$ is associative.

		To prove the distributivity property of $\oplus$ with respect to suprema, we only need to show it for one side because of the commutativity of the operation. For any $\{g_i\}_{i\in I}\subseteq\nabla$,
		\begin{align*}
			\left(f\oplus \bigvee_{i\in I}^{\leq^\text{\normalfont op}}g_i\right)(t)&=\bigvee_{r+s=t}^{\leq^\text{\normalfont op}}\left(f(r)+\bigvee_{i\in I}g_i(s)\right)=\bigvee_{r+s=t}^{\leq^\text{\normalfont op}}\bigvee_{i\in I}^{\leq^\text{\normalfont op}}f(r)+ g_i(s)\\&=\bigvee_{i\in I}^{\leq^\text{\normalfont op}}\bigvee_{r+s=t}^{\leq^\text{\normalfont op}}f(r)+ g_i(s)=\left(\bigvee_{i\in I}(f\oplus g_i)\right)(t).
		\end{align*}
		
		Finally, since $f\leq^\text{\normalfont op}\mathbf{0}$ for all $f\in\nabla$ then $\mathbf{0}=\top$ so the quantale $(\nabla,\leq^\text{\normalfont op},\oplus)$ is integral.
		
		Additionally, a routine check shows that $f\oplus g\in\nabla_L$ for every $f,g\in\nabla_L.$ Since $(\nabla_L,\leq^\text{\normalfont op})$ is a sublattice of $(\nabla,\leq^\text{\normalfont op})$ with the same top and bottom, then $(\nabla_L,\leq^\text{\normalfont op},\oplus)$ is also a CI-quantale.
	\end{proof}

	We arrive at the main result of the paper that proves that the category of quasi-pseudometric modular spaces with nonexpansive maps is isomorphic to the category of $\nabla$-categories.

	\begin{theorem}\label{thm:iso}
		$\nabla$-$\mathsf{Cat}$ is isomorphic to  $\mathsf{QPMod}_n$.
	\end{theorem}
	
	\begin{proof}
		Let us define $\mathscr{E}_\nabla:\nabla\text{-}\mathsf{Cat}\to \mathsf{QPMod}_n$ leaving morphisms unchanged and $\mathscr{E}_\nabla ((X,a))=(X,w_a)$ for every $\nabla$-category $(X,a)$, where $w_a(t,x,y)=a(x,y)(t)$ for all $x,y\in X$, $t>0.$ It is clear that $\mathscr{E}_\nabla$ is a functor.
		
		On the other hand, consider $\mathscr{E}_{\mathsf{Mod}}:\mathsf{QPMod}_n\to\nabla\text{-}\mathsf{Cat}$ leaving morphisms unchanged and $\mathscr{E}_{\mathsf{Mod}} ((X,w))=(X,a_w)$ for every quasi-pseudometric modular space $(X,w)$, where $a_w(x,y)(t)=w(t,x,y)$ for all $x,y\in X$, $t>0.$ It is straightforward to check that $(X,a_w)$ is a $\nabla$-category (notice that, due to Proposition \ref{prop:isotone}, $a_w(x,y)$ is isotone for every $x,y\in X$).
		
		To check (QC1), notice that for any $x\in X$, $a_w(x,x)(t)=w(t,x,x)=0$ by  (M1). Hence, $a_w(x,x)=\mathbf{0}$, where $\mathbf{0}$ is the unit element of $\nabla$.
		
		We next prove (QC2). Take any $x,y,z\in X$. Then
		\begin{align*}
			(a_w(x,z)\oplus a_w(z,y))(t)&=\bigvee_{r+s=t}^{\leq^\text{\normalfont op}}(a_w(x,z)(r)+a_w(z,y)(s))\\&=\bigvee_{r+s=t}^{\leq^\text{\normalfont op}}(w(r,x,z)+w(s,z,y))\leq^\text{\normalfont op}\\&\leq^\text{\normalfont op}\bigvee_{r+s=t}^{\leq^\text{\normalfont op}}w(r+s,x,y)=w(t,x,y)=a_w(x,y)(t),
		\end{align*}
		for all $t>0.$
		Consequently, $a_w(x,z)\oplus a_w(z,y)\leq^\text{\normalfont op} a_w(x,y)$.
		
		Moreover, it is obvious that a nonexpansive function between quasi-pseudometric modular spaces is a $\nabla$-functor between their corresponding $\nabla$-categories.
		
		Finally, it easily follows that $\mathscr{E}_{\mathsf{Mod}}\circ \mathscr{E}_\nabla=\mathscr{I}_{\nabla\text{-}\mathsf{Cat}}$ and $ \mathscr{E}_\nabla\circ \mathscr{E}_{\mathsf{Mod}}=\mathscr{I}_{\mathsf{QPMod}_n}$
	\end{proof}

	If we change the quantale $\nabla$ by $\nabla_L$, we obtain the category of left-continuous quasi-pseudometric modular spaces.
	
	\begin{theorem}\label{prop: modular LQcategory}
		$\nabla_L$-$\mathsf{Cat}$ is isomorphic to  $\mathsf{LQPMod}_n$.
	\end{theorem}
	
	\begin{proof}
		Let us define $\mathscr{E}_{\nabla_L}$ as the restriction of the functor $\mathscr{E}_{\nabla}$ to the category $\nabla_L\text{-}\mathsf{Cat}.$ Notice that in this case $\mathscr{E}_{\nabla_L}((X,a))=(X,w_a)$ is a left-continuous quasi-pseudometric modular space. In fact, $$w_a(t,x,y)=a(x,y)(t)=\bigvee^{\leq^\text{\normalfont op}}_{0<s<t} a(x,y)(s)=w_a(s,x,y)$$ for every $x,y\in X$ and $t>0.$ Consequently, $\mathscr{E}_{\nabla_L}:\nabla_L\text{-}\mathsf{Cat}\to \mathsf{LQPMod}_n$ is well-defined.
		
		The rest of the proof is similar to the previous one.
	\end{proof}
	
	\begin{theorem}
		
		The following diagram commutes:
		
		\begin{center}
			\begin{tikzpicture}
				\path (0,4) node (x) {\begin{minipage}{2cm}\centering $\nabla$-$\mathsf{Cat}$\end{minipage}}
				(5,4) node (x) {\begin{minipage}{2cm}\centering $\mathsf{QPMod}_n$\end{minipage}}
				
				(0,2) node (x) {\begin{minipage}{2cm}\centering $\nabla_L$-$\mathsf{Cat}$\end{minipage}}
				(5,2) node (y) {\begin{minipage}{2cm}\centering  $\mathsf{LQPMod}_n$\end{minipage}}           ;
				\draw[->] (1.,4) --node[above]{$\mathscr{E}_\nabla$} (4.1,4);
				\draw[->] (1.,2) --node[above]{$\mathscr{E}_{\nabla_L}$} (4.1,2);
				\draw[->] (0,3.7)--node[right]{$\mathscr{U}$} (0,2.3);
				\draw[->] (5,3.7)--node[right]{$\mathscr{L}$} (5,2.3);
			\end{tikzpicture}
		\end{center}
		where $\mathscr{U}:\nabla$-$\mathsf{Cat}\to\nabla_L$-$\mathsf{Cat}$ is the functor given by $\mathscr{U}((X,a))=(X,\widetilde{a})$ and leaving morphisms unchanged, where $\widetilde{a}$ is given by
		$$\widetilde{a}(x,y)(t)=\bigwedge_{0<s<t} a(x,y)(s).$$
		
	\end{theorem}
	
	\begin{proof}
		It is straightforward to check that $\mathscr{U}$ is a functor. By Remark \ref{rem:functor}, $\mathscr{L}$ is a functor. 
		
		Moreover, 
		
		$$w_{\widetilde{a}}(t,x,y)=\widetilde{a}(x,y)(t)=\bigwedge_{0<s<t} a(x,y)(s)=\bigwedge_{0<s<t}w_a(s,x,y)=\widetilde{w}_a(t,x,y)$$
		for all $x,y\in X$ and all $t>0.$ Hence
		
		$$(\mathscr{E}_{\nabla_L}\circ\mathscr{U})(X,a)=(X,w_{\widetilde{a}})=(X,\widetilde{w_a})=(\mathscr{L}\circ\mathscr{E}_\nabla)(X,a)$$
		which proves the commutativity of the diagram.
	\end{proof}
	
	We finish the paper by showing that the isomorphism between the categories  $\nabla$-$\mathsf{Cat}$ and $\mathsf{QPMod}_n$ also behaves well with respect to topology.
	
	\begin{theorem}\label{thm:topology}
		Let $(X,w)$ be a quasi-pseudometric modular space. Then $\mathcal{T}(w
		)=\mathcal{T}(a_w).$
	\end{theorem}
	
	\begin{proof}
		Let $G\in\mathcal{T}(w).$ Then given $x\in G$ there exists $t,\varepsilon>0$ such that $x\in W_{t,\varepsilon}(x)\subseteq G.$
		Define $f_{t,\varepsilon}:(0,+\infty)\to [0,+\infty]$ by
		$$f_{t,\varepsilon}(s)=\begin{cases} +\infty&\text{ if }0<s<t\\
			\varepsilon&\text{ if }t\leq s
		\end{cases}.
		$$
		
		By Lemma \ref{lem:wbnabla}, $f_{t,\varepsilon}\lhd^\text{\normalfont op}\mathbf{0}.$ We next show that $x\in B_{a_w}(x,f_{t,\varepsilon})\subseteq W_{t,\varepsilon}(x)\subseteq G.$ If $y\in B_{a_w}(x,f_{t,\varepsilon})$ then $f_{t,\varepsilon}\lhd^\text{\normalfont op}a_w(x,y).$ In particular, $f_{t,\varepsilon}(t)=\varepsilon<^\text{\normalfont op}a_w(x,y)(t)=w(t,x,y)$ so $y\in W_{t,\varepsilon}(x).$

		Conversely, let $O\in\mathcal{T}(a_w)$ and $x\in O.$ Then we can find $g\in\nabla$ with $g\lhd^\text{\normalfont op}\mathbf{0}$ such that $B_{a_w}(x,g)\subseteq G.$ By Lemma \ref{lem:wbnabla} we know that $\varepsilon:=\bigvee^{\leq^\text{\normalfont op}}\{g(t):t\in (0,+\infty)\}<^\text{\normalfont op} 0$ and there exists $s\in (0,+\infty)$ such that $g(t)=\infty$ for every $t\in (0,s).$ Let $0<t_0<s$. We assert that $x\in W_{t_0,\varepsilon}(x)\subseteq B_{a_w}(x,g)\subseteq G.$ In fact, let $y\in W_{t_0,\varepsilon}(x)$, that is, $w(t_0,x,y)<\varepsilon.$ Then $g\leq^\text{\normalfont op} f_{t_0,\varepsilon}\lhd^\text{\normalfont op}w(\_,x,y)=a_w(x,y)$ so $y\in B_{a_w}(x,f).$
	\end{proof}



\end{document}